\documentclass[a4paper,11 pt]{amsart}
\usepackage{amsfonts}
\usepackage{amssymb}
\usepackage[utf8]{inputenc}
\usepackage{amsmath}
\usepackage{pdflscape}
\usepackage{graphicx}
\usepackage[colorlinks=true, linkcolor=red, citecolor=blue]{hyperref}
\usepackage[]{epsfig}
\usepackage[]{pstricks}
\usepackage{tikz}

\newcommand{\R}{{\mathbb R}}

\newtheorem{theorem}{Theorem}[section]

\newtheorem{proposition}[theorem]{Proposition}

\theoremstyle{definition}
\newtheorem{definition}[theorem]{Definition}

\theoremstyle{remark}

\makeatletter
\@namedef{subjclassname@2020}{\textup{2020} Mathematics Subject Classification}
\makeatother

\newcommand{\remove}[1]{ }

\def\R{\mathbb R}
\def\be{\begin{equation}}
\def\ee{\end{equation}}
\def\ba{\begin{eqnarray}}
\def\ea{\end{eqnarray}}

\numberwithin{equation}{section}
\begin{document}
\title[Periodicity for Kawahara equation]{Massera's theorems for a higher order dispersive system}
\author[Capistrano--Filho]{Roberto de A. Capistrano--Filho*}
\thanks{*Corresponding author: roberto.capistranofilho@ufpe.br}
\address{Departamento de Matemática, Universidade Federal de Pernambuco (UFPE), 50740-545, Recife-PE, Brazil.}
\email{roberto.capistranofilho@ufpe.br}
\author[de Jesus]{Isadora Maria de Jesus}
\address{Instituto de Matemática, Universidade Federal de Alagoas (UFAL), Maceió-AL, Brazil and Departamento de Matemática, Universidade Federal de Pernambuco (UFPE), 50740-545, Recife-PE, Brazil}
\email{isadora.jesus@im.ufal.br; isadora.jesus@ufpe.br}

\begin{abstract}
This work is devoted to present Massera-type theorems for the Kawahara system, a higher order dispersive equation, posed in a bounded domain.  Precisely,  thanks to some properties of the semigroup and the decays of the solutions of this equation,  we are able to prove its solutions are periodic, quasi-periodic and almost periodic. 
\end{abstract}
\maketitle

\section{Introduction}
\subsection{Problem under consideration} 
Results of the existence of periodic solutions for differential equations dates back to the 50s, when in 1950, J.  L.  Massera published a remarkable paper \cite{massera} on the existence of periodic solutions to ordinary differential equations (ODE) with periodic right-hand sides. Precisely, the corresponding linear setup Massera's theorem is as follows: Consider ODE of the form
\begin{equation}\label{eq2f}
\dot{x}=A(t) x+b(t), \quad x \in \mathbb{R}^{m},
\end{equation}
with the matrix $A(t)$ and the vector $b(t)$ continuous on $\mathbb{R}_{+}$ and periodic with the same period $\tau$.  Then, the system \eqref{eq2f} has periodic solution with period $\tau$ if and only if it has a bounded solution on $\mathbb{R}_{+}$.    

So,  in this context,  we are interested to prove some periodic properties of the following Kawahara equation in a bounded domain
\begin{equation}\label{kaw} 
\begin{cases}
u_t+u_{xxx}-u_{xxxxx}+uu_x=0& (x,t) \in I\times \R\\
u(0,t)=\varphi(t), \ u(1,t)=u_x(1,t)=u_x(0,t)=0,& t\in \R\\
u_{xx}(1,t)=\alpha u_{xx}(0,t), & t\in \R\\
u(x,0)=u_0(x), &x\in I,
\end{cases}
\end{equation}
with a boundary force $\varphi(t)$ in a bounded domain $I=(0,1)$ and a damping term $\alpha u_{xx}(0,t)$, where $|\alpha|<1$.   Precisely, we are interested to understand if the system \eqref{kaw} has good properties when we investigate its solutions, considering the context introduced to Massera.  Roughly speaking,  we are interested in the study of the existence and qualitative property of recurrent solutions. This kind of property may be reformulated on the following question. 

\vspace{0.2cm}
\begin{itemize}
\item[] \textbf{Question $\mathcal{A}:$} Are there periodic solutions for the system \eqref{kaw}?
\end{itemize} 

\subsection{Physical motivation} 
Under suitable assumption on amplitude, wavelength, wave steepness and so on, the properties of the asymptotic models for water waves has been extensively studied to understand the full water wave system\footnote{ See for instance  \cite{ASL, BLS, Lannes} and references therein, for a rigorous justification of various asymptotic models for surface and internal waves.}.  In this spirit, formulating the waves as a free boundary problem of the incompressible, irrotational Euler equation in an appropriate non-dimensional form, one has two non-dimensional parameters $\delta := \frac{h}{\lambda}$ and $\varepsilon := \frac{a}{h}$, where the water depth, the wavelength and the amplitude of the free surface are parameterized as $h, \lambda$ and $a$, respectively. Moreover, another non-dimensional parameter $\mu$ is called the Bond number, which measures the importance of gravitational forces compared to surface tension forces. 

The physical condition $\delta \ll 1$ characterizes the waves, which are called long waves or shallow water waves, but there are several long wave approximations according to relations between $\varepsilon$ and $\delta$.  So that,  if we consider $$\varepsilon = \delta^4 \ll 1\quad \text{and} \quad \mu = \frac13 + \nu\varepsilon^{\frac12},$$ and in connection with the  critical Bond number $\mu = \frac13$,  we have the so-called Kawahara equation which is an equation derived by Hasimoto and Kawahara in \cite{Hasimoto1970,Kawahara} that take the form \[\pm2 u_t + 3uu_x - \nu u_{xxx} +\frac{1}{45}u_{xxxxx} = 0.\]
Rescaling this equation,  we will study in this paper the following system
\[u_t + uu_x + u_{xxx} -u_{xxxxx} = 0.\]

\subsection{Historical background} Before answering the Question $\mathcal{A}$,  let us introduce a state of the art related to the Kawahara equation.  As mentioned before, problems related with higher order dispersive systems are extensively studied.  Precisely, stabilization and control problems have been studied in recent years.

A pioneer work is due to Silva and Vasconcellos \cite{CV}. The authors studied the stabilization of global solutions of the linear Kawahara equation in a bounded interval under the effect of a localized damping mechanism.  The second work in this way is due Capistrano-Filho \textit{et al.} \cite{CaKawahara}.  In this work the authors considered the Kawahara equation in a bounded domain $Q_T = (0, T) \times (0,L)$,
\begin{equation}\label{int1}
\left\lbrace
\begin{array}{llr}
u_{t} + u_{x} +u_{xxx} - u_{xxxxx}+uu_x= f(t,x) & \mbox{in} \ Q_{T}, \\ 
u(t,0)=u(t,L)=  u_{x}(t,0)=u_{x}(t,L)=u_{xx}(t,L) = 0 & \mbox{on} \ [0,T],\\
 u(0,x) = u_{0}(x) & \mbox{in} \ [0,L].
\end{array}\right. 
\end{equation} 
In this article the authors were able to introduce an internal feedback law in \eqref{int1}, considering general nonlinearity $u^pu_x$, $p\in [1,4)$, instead of $uu_x$.  They proved that under the effect of the damping mechanism the energy associated with the solutions of the system decays exponentially. 

Related with internal control issues in bounded domain, Chen \cite{MoChen} presented results considering the Kawahara equation \eqref{int1} posed on a bounded interval with a distributed control $f(t,x)$ and homogeneous boundary conditions. She showed the result taking advantage of a Carleman estimate associated with the linear operator of the Kawahara equation with an internal observation. With this in hand, she was able to get a null controllable result when $f$ is effective in a $\omega\subset(0,L)$.  In \cite{CaGo},  considering the system \eqref{int1} with an internal control $f(t,x)$ and homogeneous boundary conditions, the authors are able to show that the equation in consideration is exact controllable in $L^2$-weighted Sobolev spaces and, additionally, the Kawahara equation is controllable by regions on $L^2$-Sobolev space.

Recently,  a new tool to find control properties for the Kawahara operator was proposed in \cite{CaSo, CaSoGa}.  First, in \cite{CaSo}, the authors showed a new type of controllability for the Kawahara equation, what they called \textit{overdetermination control problem}. They are able to find a control acting at the boundary that guarantees that the solution of the problem under consideration satisfies an integral condition. In addition, when the control acts internally in the system, instead of the boundary, the authors proved that this condition is also satisfied.   After that, in \cite{CaSoGa}, the authors extend this idea for the internal control problem for the Kawahara equation on unbounded domains.  Precisely, under certain hypotheses over the initial and boundary data,  they are able to prove that there exists an internal control input such that solutions of the Kawahara equation satisfies an integral overdetermination condition considering the Kawahara equation posed in the real line, left half-line and right half-line.  

We finish presenting the last works in control theory related to the Kawahara equation. In \cite{boumediene,  CaVi}, under suitable assumptions on the time delay coefficients, the authors are able to prove that solutions of the Kawahara system are exponentially stable. The results are obtained using Lyapunov approach and compactness-uniqueness argument.  We caution that this is only a small sample of the extant work on control theory for the Kawahara equation. 

\subsection{Notation and auxiliary results}  Before presenting the main results of the article, let us introduce some notation and two auxiliary results. Denote by $C^1(\R)$ a function space whose elements are continuously differentiable complex-valued functions on $\R$ and its norm
$$
\|u\|_{C^1(\R)}:=\|u\|_{C(\R)}+\|u_t\|_{C(\R)}, \quad \forall u\in C^1(\R).
$$
Set
$$C^1_b(\R)=\{u\in C^1(\R);\|u\|_{C^1(\R)}<+\infty\}.$$
Consider by $L^2(I)$ the space of all Lebesgue square integrable complex-valued functions on $I$ with the following inner product 
$${\displaystyle (u,v):=\mbox{\rm Re}\left(\int_0^1 u\overline{v}dx\right), \quad \forall u,v\in L^2(I),}$$ 
where $\overline{v}$ denotes the conjugate of $v$. With the previous inner product, we define in $L^2(I)$ the following norm 
$${\displaystyle  \|u\|=(u,u)^{\frac{1}{2}}}.$$
Additionally,  let $H^s(I),s\geq 0,$ be the classical Sobolev spaces of complex-valued functions on $I$ with its the classical inner produc and norm, denoted by $\|\cdot\|_{H^s(I)}.$  Finally,  consider  $$H^s_\alpha(I)=\{u\in H^s(I)|u^{(5i)}(0)=0=u^{(5i)}(1),u^{(5i+1)}(0)=0=u^{(5i+1)}(1), u^{(5i+2)}(1)=\alpha u^{(5i+2)}(0)\}$$
where the derivatives are of order less than or equal to $n-1.$ The norm and inner product of  $H^s_\alpha(I)$ are inherited from $H^s(I).$

The first result is devoted to prove the well-posedness \textit{via} semigroup theory,  which is the key to prove the other main results of the article.  Precisely,  we first prove that the linear Kawahara operator generates $\{S(t)\}_{t\geq0}$ the $C_0$--semigroup  of contraction on $L^2(I)$.  
\begin{theorem}
\label{main1}
There exists $\omega>0$  such that for any $k=0,1,2,3,4$ and $5$, we can find a positive constant $C_k>0$ which the semigroup associated to the linear Kawahara operator satisfies
$$\|S(t)u_0\|_{H^k_\alpha(I)}\leq C_ke^{-\omega t}\|u_0\|_{H^k_\alpha(I)}, $$
for all $t>0$.
\end{theorem}

The previous theorem is the key to prove the existence of the bounded solution for the  Kawahara equation \eqref{kaw}.  For that,  pick  $$X:=C_b(\R, H^2(I))$$ with a norm  $$ \|u\|_X:=\sup_{t\in \R}  \|u(t)\|_{H^2(I)}$$ and define the following set $$X_\rho:=\{u\in X| \|u\|_X\leq \rho\}.$$  The next theorem, thanks to the previous one,  ensures that the solutions of \eqref{kaw} are bounded.

\begin{theorem} \label{main2}There exists a constant $\epsilon>0$ such that for all $\varphi\in C^1(\R)$ satisfying  $\|\varphi\|_{C^1(\R)}\leq \epsilon,$ the system \eqref{kaw} admits a unique solution $u$ such that $$\|u\|_X\leq C\epsilon,$$ where $C>0$ is a constant independent of $\epsilon.$
\end{theorem}

\subsection{Massera-type theorems and structure of the article} With the previous background in hand it is clear that no results concerning the recurrent solutions for the Kawahara system are presented in the literature.  This manuscript is interesting to fill this gap giving answers for the Question $\mathcal{A}$ before presented.   Precisely,  the next three theorems give us Massera-type theorems for a higher order dispersive system  that is,  the result below ensures that solution of \eqref{kaw} is $T$--periodic.
\begin{theorem}\label{main3}
Let $$\|\varphi\|_{C^1(\R)}\leq \epsilon,$$ where $\epsilon$ is the constant determined by Theorem \ref{main2}.  If $\varphi$ is a function $T$--periodic, thus $u$ solution of \eqref{kaw}, given by Theorem \ref{main2},  is also a function $T$--periodic.
\end{theorem}

Additionally,  the next Massera-type theorem gives some property of the periodicity of the solution to \eqref{kaw}.  The result can be read as follows.
\begin{theorem}\label{main4}
Let $$\|\varphi\|_{C^1(\R)}\leq \epsilon,$$ where $\epsilon$ is the constant determined by Theorem \ref{main2}.  If $\varphi$ is quasi-periodic function, the solution $u$ of \eqref{kaw}, obtained in Theorem \ref{main2}, is also quasi-periodic function. Moreover,  if $\varphi$ is $\overline{\omega}$--quasi-periodic function in $t$, thus  the solution $u$ of \eqref{kaw}, obtained in Theorem \ref{main2}, is also $\overline{\omega}$--quasi-periodic function in $t.$
\end{theorem}

Finally,  let us present the last result of this work. Precisely, we are able to prove that the solutions of \eqref{kaw} are almost periodic.

\begin{theorem}\label{main5}
Let $\|\varphi\|_{C^1(\R)}\leq \epsilon,$ where $0<\epsilon\ll 1$  is obtained via Theorem \ref{main2}.  If  $\varphi, \varphi'$ are functions almost periodic,  the solution $u$ of \eqref{kaw},  given by Theorem \ref{main2}, is also almost periodic function. 
\end{theorem}

The remainder of the paper is organized as follows.  In the Section \ref{Sec2}, we present the auxiliary results that are essential for the proof of the Massera-type theorems, precisely, we present the proof of Theorems \ref{main1} and \ref{main2}.  After that, in the Section \ref{Sec3}, we present the answer for the question $\mathcal{A}$ which is divided in three results, that is, we present the proof of Theorems \ref{main3}, \ref{main4} and \ref{main5}.  Further comments are presented in Section \ref{Sec4}. Finally,  on the Appendix \ref{appendix}, we give some properties of the energy associated to \eqref{kaw}.

\section{Preliminaries }\label{Sec2}
In this section we are interested to prove some properties of the following linear Kawahara system
\begin{equation}\label{kaw_1} 
\begin{cases}
u_t+u_{xxx}-u_{xxxxx}=0& (x,t)\in I\times \R\\
u(0,t)=u(1,t)=u_x(1,t)=u_x(0,t)=0,& t\in \R\\
u_{xx}(1,t)=\alpha u_{xx}(0,t), & t\in \R\\
u(x,0)=u_0(x), &x\in I.
\end{cases}
\end{equation}
which are essential for the rest of the article. 

\subsection{Well-posedness: Linear system}
From now on $C$ with or without subscripts denotes positive constants whose value may change on different occasions. We will write the dependence of constant on parameters explicitly if it is essential.  Additionally,  we denote $\lambda_* >0$ is the smallest constant such that the following inequality holds
$$\|u_x\|^2\geq \lambda_* \|u\|^2,\ u\in H^1_0(I).$$

Consider the following operator $A:D(A)\subset L^2(I)\longrightarrow L^2(I)$ defined by 
$$Au:=-u_{xxx}+u_{xxxxx}$$ where $$D(A)=\{u\in H^5(I): u(0)=u(1)=u_x(0)=u_x(1)=0, u_{xx}(1)=\alpha u_{xx}(0)\},$$ with $|\alpha|<1, $ and its adjoint $A^*v=v_{xxx}-v_{xxxxx}$ with$$D(A^*)=\{v\in H^5(I): v(0)=v(1)=v_x(0)=v_x(1)=0, v_{xx}(0)=\alpha v_{xx}(1)\}.$$
Thus, the following property holds true.
\begin{proposition}\label{c_0}
$A$ generates a $C_0$ semigroup of contractions on $L^2(I).$
\end{proposition}
\begin{proof}
Since $A$ is a continuous linear operator, using the closed graph theorem,  $A$ has the closed graph. Moreover, as $D(A)$ is dense in $L^2(I)$,  if we prove that $A$ and $A^*$
are dissipative,  thanks to \cite[Corollary 4.4]{Pazy} we have that $A$ generates a $C_0$ semigroup of contractions on $L^2(I).$ To do this, by using the definitions of $A$ and $A^*$, we get, integrating by parts that
$$
(Au,u)=\dfrac{1}{2}(\alpha^2-1)(u_{xx}(0))^2\leq 0
$$
and
$$
(A^*v,v)=\dfrac{1}{2}(\alpha^2-1)(v_{xx}(1))^2\leq 0,
$$
that is,  $A$ and $A^*$ are dissipative,  and so the proof is finished. 
\end{proof}


From now on, denote by $\{S(t)\}_{t\geq 0}$ the $C_0-$semigroup associated with $A,$ so $u(t)=S(t)u_0$  is the mild solution of the linearized system \eqref{kaw_1}. Next result ensure some properties of the solution of the linear Kawahara system.

\begin{proposition}
\label{prop2} Let $u$ solution of \eqref{kaw_1}. Then, we have for all $T>0$ that
\begin{enumerate}
\item[$(i)$] $\|u(\cdot, T)\|\leq \|u_0\|;$
\item[$(ii)$]${\displaystyle (1-\alpha^2)\int_0^T u_{xx}^2(0,t)dt\leq \|u_0\|^2};$
\item[$(iii)$] $\|u\|_{L^2(0,T;H^2(I))}\leq \sqrt{\dfrac{1}{3}\left(\dfrac{1}{1-\alpha^2}+4T\right)} \|u_0\|.$
\end{enumerate}
\end{proposition}
\begin{proof}
Since $\{S(t)\}_{t\geq 0}$ is a $C_0-$semigroup of contractions, item (i) follows.  Now, observe that 
$$\dfrac{d}{dt}\left(\|u(t)\|^2_{L^2(I)}\right)=2(u_t,u)=2(Au,u)=(\alpha^2-1)(u_{xx}(0))^2,$$ and integrating in $(0,T)$ we get (ii).  

Now,  multiplying \eqref{kaw_1} by $xu$, integrating by parts and using the boundary condition we have that
\begin{equation*}
\begin{split}
 \dfrac{d}{dt}\left(\int_I xu^2dx\right)=& 2\int_I xuu_tdx=-2\int_I xuu_{xxx}dx+2\int_I xuu_{xxxxx}dx\\
=& \int_I u^2dx-3\int_I (u_{x})^2dx+\alpha^2(u_{xx}(0,t))^2-5\int_I (u_{xx})^2dx.
\end{split}
\end{equation*}
Integrating both side in $(0,T)$,  holds that
\begin{equation*}
\begin{split} 3\int_0^T\int_I (u_{x})^2dxdt+3 \int_0^T\int_I (u_{xx})^2dxdt \leq& \int_I x(u_0)^2dx+\int_0^T\int_I u^2dxdt+ \int_0^T\alpha^2(u_{xx}(0,t))^2dt\\
\leq& \int_I(u_0)^2dx+\int_0^T\|u_0\|^2dt+ \int_0^T\alpha^2(u_{xx}(0,t))^2dt
\end{split}
\end{equation*}
since $\|u(t)\|\leq \|u_0\|$ for al  $t\geq 0. $ So, using item (ii), we obtain
$$3\int_0^T\int_I (u_{x})^2dxdt+3 \int_0^T\int_I (u_{xx})^2dxdt\leq (1+T+\dfrac{\alpha^2}{1-\alpha^2})\|u_0\|^2.$$
Therefore, 
\begin{equation*}
\begin{split}
3\|u\|^2_{H^2(I)}=&\ 3\int_0^T\int_I u^2dxdt+3\int_0^T\int_I (u_{x})^2dxdt+3 \int_0^T\int_I (u_{xx})^2dxdt\\
\leq&\ 3T\|u_0\|^2+(1+T+\dfrac{\alpha^2}{1-\alpha^2})\|u_0\|^2=\left(\dfrac{1}{1-\alpha^2}+4T\right)\|u_0\|^2,
\end{split}
\end{equation*}
showing the result. 
\end{proof}


The next result ensures the decay of the semigroup associated with the Kawahara operator. This can be proved by using the results shown in the Appendix \ref{appendix}.
\begin{proposition}
\label{prop4}
Existe $\omega>0$ e $C>0$ tal que
$$\|S(t)u_0\|\leq Ce^{-\omega t}\|u_0\|, \ t\geq 0.$$
\end{proposition}
\begin{proof}
Consider $E(t)=\dfrac{1}{2}\|u\|^2(t)$ the energy associated with \eqref{kaw_1}.
So, thanks to the Theorem \ref{teo6}, we have that 
$$\|S(t)u_0\|^2=2 E(t)\leq C\|u_0\|^2e^{-\mu t},$$
and taking the square root of both sides in the previous inequality with $\omega=\dfrac{\mu}{2}>0$ the results holds.
\end{proof}

The next result shows that the solutions of \eqref{kaw_1} are bounded. 
\begin{proposition}
\label{prop6}There exists $C>0$ such that for any $t>0$, 
$$\|S(t)u_0\|_{H^2(I)}\leq C\sqrt{\dfrac{1}{1-\alpha^2}\dfrac{1}{t}+1}\|u_0\|,$$
holds for any $u_0\in L^2(I).$
\end{proposition}
\begin{proof}
Define the following function 
$$g(t):=\int_0^t \|S(s)u_0\|_{H^2(I)}^2ds.$$
Applying the mean value theorem we have the existence of $\tau\in\left(0,\dfrac{t}{2}\right)$ such that
$$\|S(\tau)u_0\|_{H^2(I)}^2\cdot \left(\dfrac{t}{2}\right)=\int_0^\frac{t}{2} \|S(s)u_0\|_{H^2(I)}^2ds.$$
Thanks to the item (iii) of Proposition \ref{prop2},  we get
$$\|S(\tau)u_0\|_{H^2(I)}^2\cdot \left(\dfrac{t}{2}\right)=\int_0^\frac{t}{2} \|S(s)u_0\|_{H^2(I)}^2ds\leq \dfrac{1}{3}\left(\dfrac{1}{1-\alpha^2}+2t\right)\|u_0\|^2.$$
Thus, 
$$\|S(\tau)u_0\|_{H^2(I)}^2\leq \dfrac{1}{3}\left(\dfrac{1}{1-\alpha^2}+2t\right)\left(\dfrac{2}{t}\right)\|u_0\|^2=\dfrac{4}{3}\left(\dfrac{1}{1-\alpha^2}\dfrac{1}{2t}+1\right)\|u_0\|^2$$
and so, 
$$\|S(\tau)u_0\|_{H^2(I)}\leq \dfrac{2\sqrt{3}}{3}\sqrt{\dfrac{1}{1-\alpha^2}\dfrac{1}{2t}+1}\|u_0\|\leq \dfrac{2\sqrt{3}}{3}\sqrt{\dfrac{1}{1-\alpha^2}\dfrac{1}{t}+1}\|u_0\|.$$
Finally, semigroup properties ensures that
$$\|S(t)u_0\|_{H^2(I)}=\|S(t-\tau)S(\tau)u_0\|_{H^2(I)}\leq C_1\|S(\tau)u_0\|_{H^2(I)}\leq C_1\dfrac{2\sqrt{3}}{3}\sqrt{\dfrac{1}{1-\alpha^2}\dfrac{1}{t}+1}\|u_0\|,$$
and the proof is achieved. 
\end{proof}

\subsection{Proof of Theorem \ref{main1}} 
Considering $k=0$, thus the result is consequence of Proposition \ref{prop4}.  Now,  taking  $u_0\in D(A)=H^5_\alpha(I)$,  semigroup theory ensures that $u=S(t)u_0\in D(A)$ and $u_t=Au=S(t)(Au_0)$, with $Au_0\in L^2(I).$ Pick $v=u_t$, we have that $v$ satisfies the following initial value problem 
\begin{equation}
\label{eq6}
\begin{cases}
v_t=Av,& (x,t)\in I\times (0,T),\\
v(x,0)=v_0(x)=(Au_0)(x)\in L^2(I), & x\in I.
\end{cases}
\end{equation}
Proposition \ref{prop4} yields that
$$\|v(\cdot, t)\|\leq Ce^{-\omega t}\|v_0\|=Ce^{-\omega t}\|Au_0\|.$$
Since the following norms $\|u\|+\|Au\|$ and $\|u\|_{D(A)}$ are equivalents in $D(A),$ we ensure the existence of two constants $M_1, M_2> 0$ such that $$M_1\|u\|_{D(A)}\leq \|u\|+\|Au\|\leq M_2\|u\|_{D(A)}.$$
Thus, 
\begin{equation*}
\begin{split}
\|S(t)u_0\|_{H^5_\alpha(I)}\leq&\ M_1^{-1}(\|S(t)u_0\|+\|A(S(t)u_0)\|)\\
\leq&\ M_1^{-1}(Ce^{-\omega t}\|u_0\|+Ce^{-\omega t}\|Au_0\|)\\
\leq&\ M_1^{-1}CM_2e^{-\omega t}\|u_0\|_{H^5_\alpha(I)}.
\end{split}
\end{equation*}
The results for $k=1,2,3$ and $4,$ are consequences of an interpolation argument. So, Theorem \ref{main1} is shown.  \qed

\subsection{Well-posendess: Bounded solutions for the nonlinear system}
Consider the following initial boundary value problem
\begin{equation}\label{eq1}
\begin{cases}
u_t+u_{xxx}-u_{xxxxx}+uu_x=0, & (x,t)\in I\times \R,\\
u(0,t)=\varphi(t), \ u(1,t)=u_x(1,t)=u_x(0,t)=0, &  t\in \R,\\
u_{xx}(1,t)=\alpha u_{xx}(0,t), & t\in \R,\\
u(x,0)=u_0(x), &x\in I.
\end{cases}
\end{equation}
Let us study the bounded solution to the system  \eqref{eq1}.  To do this,  let us consider 
$$y(x,t):=u(x,t)+A(x)\varphi(t),$$
with the function $A$ defined by
$$A(x)=\dfrac{\alpha-1}{\alpha+1}x^2+\dfrac{2}{\alpha+1}x-1$$
and $\varphi\in C^1(\R).$ If we suppose that $u$ satisfies \eqref{eq1}, so we have that $y$ satisfies
\begin{equation}
\label{eq7}
\begin{cases}
y_t+y_{xxx}-y_{xxxxx}+yy_x+ay_x+by=f, &(x,t)\in I\times \R,\\
y(0,t)=y(1,t)=0, & t\in \R,\\
y_x(1,t)=\alpha y_x(0,t),&t\in \R,\\
y_{xx}(1,t)=\alpha y_{xx}(0,t)+\left(\dfrac{-2(\alpha-1)^2}{\alpha+1}\right)\varphi(t) &t\in  \R,
\end{cases}
\end{equation}
with 
\begin{equation}
\label{eq7a}
y_x(0,t)=\dfrac{2}{\alpha+1}\varphi(t),\quad a(x,t)=-A(x)\varphi(t), \quad b(x,t)=-A'(x)\varphi(t)
\end{equation} 
and 
\begin{equation}
\label{eq7b}
f=A(x)\varphi'(t)-A(x)A'(x)\varphi^2(t).
\end{equation}
Moreover, we consider that $y$ is a mild solution of \eqref{eq7} if satisfies the integral equation 
\begin{equation}
\label{eq8}
y(t)=S(t-r)y(r)+\int_r^tS(t-r)(-yy_x-ay_x-by+f)(s)ds,
\end{equation}
for all $t\geq r$ and each $ r\in\R$. Thus,  as $y$ is a mild solution of \eqref{eq7}, we have that $$u(x,t)=y(x,t)-A(x)\varphi(t)$$ is a mild solution of \eqref{eq1}. With this in hand, we are in position to prove our second auxiliary result.

\begin{proof}[Proof of Theorem \ref{main2}]
Straightforward calculation shows that, using integration by parts, we get 
\begin{equation}
\label{eq9}
\int_0^\infty e^{-\frac{\omega}{2}\tau}\sqrt{\dfrac{1}{\tau}}d\tau
= \int_0^\infty\omega e^{-\frac{\omega}{2}\tau}\sqrt{\tau}d\tau= \int_0^\infty\omega e^{-\frac{\omega}{4}\tau}e^{-\frac{\omega}{4}\tau}\sqrt{\tau}d\tau.
\end{equation}
Pick a function $h(\tau)=e^{-\frac{\omega}{2}\tau}\tau,$ with  $\tau\in \R.$ Since $$h'(\tau)=\left(1-\dfrac{\omega}{2}\tau\right)e^{-\frac{\omega}{2}\tau}=0\Leftrightarrow \tau=\dfrac{2}{\omega},$$ $h'(t)>0$ for $\tau<\dfrac{2}{\omega}$ and $h'(t)<0$ for $\tau>\dfrac{2}{\omega},$  yields that  
$$h(\tau)\leq h\left(\dfrac{2}{\omega}\right)=\dfrac{2e^{-1}}{\omega},  \quad \forall\tau\in \R. $$
So we have
$$e^{-\frac{\omega}{4}\tau}\sqrt{\tau}\leq \sqrt{\dfrac{2e^{-1}}{\omega}}.$$
Therefore, 
\begin{equation}
\label{eq10}
\begin{split}
\int_0^\infty e^{-\frac{\omega}{2}\tau}\sqrt{\dfrac{1}{\tau}}d\tau&=\int_0^\infty\omega e^{-\frac{\omega}{4}\tau}e^{-\frac{\omega}{4}\tau}\sqrt{\tau}d\tau\\&\leq \omega\sqrt{\dfrac{2e^{-1}}{\omega}}\int_0^\infty e^{-\frac{\omega}{4}\tau}d\tau\\&=4\sqrt{\dfrac{2e^{-1}}{\omega}}.
\end{split}
\end{equation}

By other hand,  
\begin{equation}
\label{eq11}
 \int_0^\infty e^{-\frac{\omega}{2}\tau}d\tau= \dfrac{2}{\omega}.
\end{equation}
Thus,  Theorem \ref{main1},  Proposition \ref{prop6} and Agmon inequality\footnote{Agmon inequality in one dimensional case: $\|u\|_{L^\infty(I)}\leq C\|u\|_{L^2(I)}^\frac{3}{4}\|u\|_{H^{2}(I)}^{\frac{1}{4}}\leq C\|u\|_{H^2(I)}, \ I=(0,1).$}, ensure that
\begin{equation*}
\begin{split}
\left\|\int_{-\infty}^t S(t-s)y(\cdot,s)y_x(\cdot,s)ds\right\|_{H^2(I)}\leq& \int_{-\infty}^t\left\| S\left(\dfrac{t-s}{2}\right) S\left(\dfrac{t-s}{2}\right)y(\cdot,s)y_x(\cdot,s)\right\|_{H^2(I)}ds\\
\leq& C_2\int_{-\infty}^te^{-\frac{\omega}{2}(t-s)}\left\| S\left(\dfrac{t-s}{2}\right)y(\cdot,s)y_x(\cdot,s)\right\|_{H^2(I)}ds\\
\leq& CC_2\int_{-\infty}^te^{-\frac{\omega}{2}(t-s)} \sqrt{\dfrac{1}{1-\alpha^2}\dfrac{2}{t-s}+1}\left\|y(\cdot,s)y_x(\cdot,s)\right\|ds\\
\leq& C\int_{-\infty}^te^{-\frac{\omega}{2}(t-s)} \sqrt{\dfrac{1}{1-\alpha^2}\dfrac{2}{t-s}+1}\left\|y(\cdot,s)\right\|_{H^2(I)}\left\|y_x(\cdot,s)\right\|ds\\
\leq& C\left\|y\right\|_{X}^2\int_{-\infty}^te^{-\frac{\omega}{2}(t-s)} \sqrt{\dfrac{1}{1-\alpha^2}\dfrac{1}{t-s}+1}ds.
\end{split}
\end{equation*}
Thanks to the previous inequality,  taking $\tau=t-s$, we get
\begin{equation}
\label{eq12}
\begin{split}
\left\|\int_{-\infty}^t S(t-s)y(\cdot,s)y_x(\cdot,s)ds\right\|_{H^2(I)}\leq&C\left\|y\right\|_{X}^2\int_0^{\infty}e^{-\frac{\omega}{2}\tau} \left(\sqrt{\dfrac{1}{1-\alpha^2}}\sqrt{\dfrac{1}{\tau}}+\sqrt{1}\right)d\tau\\
\leq&C\left\|y\right\|_{X}^2\int_0^{\infty}e^{-\frac{\omega}{2}\tau} \left(\sqrt{\dfrac{1}{\tau}}+1\right)d\tau\\ \leq& C\left(\dfrac{1}{\sqrt{\omega}}+\dfrac{1}{\omega}\right)\left\|y\right\|_{X}^2.
\end{split}
\end{equation}
Now,  considering $\varphi\in C^1(\R),$  we have $a,b\in X$. Therefore,  we get,  using the same computations as before, that
\begin{equation}
\label{eq14}
\begin{split}
\left\|\int_{-\infty}^t S(t-s)a(\cdot,s)y_x(\cdot,s)ds\right\|_{H^2(I)}\leq& C\left\|a\right\|_{X}\left\|y\right\|_{X}\int_0^{\infty}e^{-\frac{\omega}{2}\tau} \left(\sqrt{\dfrac{1}{\tau}}+1\right)d\tau\\
\leq &C\left(\dfrac{1}{\sqrt{\omega}}+\dfrac{1}{\omega}\right)\left\|a\right\|_{X}\left\|y\right\|_{X}
\end{split}
\end{equation}
and
\begin{equation}
\label{eq15}
\begin{split}
 \left\|\int_{-\infty}^t S(t-s)b(\cdot,s)y(\cdot,s)ds\right\|_{H^2(I)} \leq& C\left\|b\right\|_{X}\left\|y\right\|_{X}\int_0^{\infty}e^{-\frac{\omega}{2}\tau} \left(\sqrt{\dfrac{1}{\tau}}+1\right)d\tau\\
 \leq &  C\left(\dfrac{1}{\sqrt{\omega}}+\dfrac{1}{\omega}\right)\left\|b\right\|_{X}\left\|y\right\|_{X}.
\end{split}
\end{equation}

Using the change of variable $y$ by $y-z$ in \eqref{eq14} and \eqref{eq15}, respectively, yields that
\begin{equation}
\label{eq16}
\left\|\int_{-\infty}^t S(t-s)a(\cdot,s)\left[y(\cdot,s)-z(\cdot,s)\right]_xds\right\|_{H^2(I)}\leq C\left(\dfrac{1}{\sqrt{\omega}}+\dfrac{1}{\omega}\right)\left\|a\right\|_{X}\left\|y-z\right\|_{X}
\end{equation}
and
\begin{equation}
\label{eq17}
\left\|\int_{-\infty}^t S(t-s)b(\cdot,s)\left[y(\cdot,s)-z(\cdot,s)\right]ds\right\|_{H^2(I)}\leq C\left(\dfrac{1}{\sqrt{\omega}}+\dfrac{1}{\omega}\right)\left\|b\right\|_{X}\left\|y-z\right\|_{X}.
\end{equation}
Additionally,  thanks to Theorem \ref{main1}, we have
\begin{equation}
\label{eq18}
\begin{split}
\left\|\int_{-\infty}^t S(t-s)f(\cdot,s)ds\right\|_{H^2(I)}\leq&\int_{-\infty}^t\left\| S\left(t-s\right)f(\cdot,s)\right\|_{H^2(I)}ds\\
\leq& C_2\int_{-\infty}^te^{-\omega(t-s)}\left\| f(\cdot,s)\right\|_{H^2(I)}ds\\
\leq& C\left\|f\right\|_{X}\int_{-\infty}^te^{-\omega(t-s)} ds\\
=& C\left\|f\right\|_{X}\int_{0}^\infty e^{-\omega\tau} ds= \dfrac{C}{\omega}\left\|f\right\|_{X},
\end{split}
\end{equation}
where in the last line we have used the following change of variable $\tau=t-s$.  Also, note that by analogous process yields that
\begin{equation}
\label{eq19}
\begin{split}
 &\left\|\int_{-\infty}^t S(t-s)\left[y(\cdot,s)y_x(\cdot,s)-z(\cdot,s)z_x(\cdot,s)\right]ds\right\|_{H^2(I)}\\
 \leq&\int_{-\infty}^t\left\| S\left(\dfrac{t-s}{2}\right) S\left(\dfrac{t-s}{2}\right)\left[y(\cdot,s)y_x(\cdot,s)-z(\cdot,s)z_x(\cdot,s)\right]\right\|_{H^2(I)}ds\\
\leq& C\int_{-\infty}^te^{-\frac{\omega}{2}(t-s)} \sqrt{\dfrac{1}{1-\alpha^2}\dfrac{1}{t-s}+1}\left\|\dfrac{1}{2}\dfrac{d}{dx}(y^2-z^2)(\cdot,s)\right\|ds\\
\leq& C\int_{-\infty}^te^{-\frac{\omega}{2}(t-s)} \sqrt{\dfrac{1}{1-\alpha^2}\dfrac{1}{t-s}+1}\left\|(y-z)(\cdot,s)\right\|_{H^1(I)}\left\|(y+z)(\cdot,s)\right\|_{H^1(I)}ds\\
\leq& C\left\|y-z\right\|_{X}\left\|y+z\right\|_{X}\int_{-\infty}^te^{-\frac{\omega}{2}(t-s)} \sqrt{\dfrac{1}{1-\alpha^2}\dfrac{1}{t-s}+1}ds\\
\leq&C\left(\dfrac{1}{\sqrt{\omega}}+\dfrac{1}{\omega}\right)\left\|y-z\right\|_{X}\left\|y+z\right\|_{X}.
\end{split}
\end{equation}
Set
\begin{equation}
\label{eq20}
(\Psi y)(t):={\displaystyle \int_{-\infty}^t S(t-s)(-yy_x-ay_x-by+f)(\cdot,s)ds}.
\end{equation}
Now, using \eqref{eq12}, \eqref{eq14}, \eqref{eq15} and \eqref{eq18}, we get that
\begin{equation}
\label{eq21}
\begin{split}
\|\Psi y\|_{X}=&\sup_{t\in\R}\|(\Psi y)(t)\|_{H^2(I)}\leq\|(\Psi y)(t)\|_{H^2(I)}\\
\leq& \left\|\int_{-\infty}^t S(t-s)(yy_x)(\cdot,s)ds\right\|_{H^2(I)}+\left\|\int_{-\infty}^t S(t-s)(ay_x)(\cdot,s)ds\right\|_{H^2(I)}\\
&+ \left\|\int_{-\infty}^t S(t-s)(by)(\cdot,s)ds\right\|_{H^2(I)}+\left\|\int_{-\infty}^t S(t-s)f(\cdot,s)ds\right\|_{H^2(I)}\\
\leq&C\left(\dfrac{1}{\sqrt{\omega}}+\dfrac{1}{\omega}\right)\left\|y\right\|_{X}^2+C\left(\dfrac{1}{\sqrt{\omega}}+\dfrac{1}{\omega}\right)\left\|a\right\|_{X}\left\|y\right\|_{X}\\
&+C\left(\dfrac{1}{\sqrt{\omega}}+\dfrac{1}{\omega}\right)\left\|b\right\|_{X}\left\|y\right\|_{X}+\dfrac{C}{\omega}\left\|f\right\|_{X}.
\end{split}
\end{equation}
Thanks to \eqref{eq16}, \eqref{eq17} and \eqref{eq19},  we get that
\begin{equation}
\label{eq23}
\begin{split}
\|\Psi y-\Psi z\|_{X}=&\sup_{t\in\R}\|(\Psi y-\Psi z)(t)\|_{H^2(I)}\leq \|(\Psi y-\Psi z)(t)\|_{H^2(I)}\\
\leq& \left\|\int_{-\infty}^t S(t-s)(yy_x-zz_x)(\cdot,s)ds\right\|_{H^2(I)}\\
&+\left\|\int_{-\infty}^t S(t-s)\left(a(y-z)_x\right)(\cdot,s)ds\right\|_{H^2(I)}\\
&+ \left\|\int_{-\infty}^t S(t-s)\left(b(y-z)\right)(\cdot,s)ds\right\|_{H^2(I)}\\
\leq&C\left(\dfrac{1}{\sqrt{\omega}}+\dfrac{1}{\omega}\right)\left\|y-z\right\|_{X}\left\|y+z\right\|_{X}+C\left(\dfrac{1}{\sqrt{\omega}}+\dfrac{1}{\omega}\right)\left\|a\right\|_{X}\left\|y-z\right\|_{X}\\
&+C\left(\dfrac{1}{\sqrt{\omega}}+\dfrac{1}{\omega}\right)\left\|b\right\|_{X}\left\|y-z\right\|_{X}.
\end{split}
\end{equation}
Moreover,  we have that
\begin{equation}
\label{eq25}
\|\Psi y\|_{X}\leq C\left(\dfrac{1}{\sqrt{\omega}}+\dfrac{1}{\omega}\right)\rho^2+C\left(\dfrac{1}{\sqrt{\omega}}+\dfrac{1}{\omega}\right)(\left\|a\right\|_{X}+\left\|b\right\|_{X})\rho+\dfrac{C}{\omega}\left\|f\right\|_{X}
\end{equation}
and
\begin{equation}
\label{eq26}
\|\Psi y-\Psi z\|_{X}\leq C\left(\dfrac{1}{\sqrt{\omega}}+\dfrac{1}{\omega}\right)\rho\left\|y-z\right\|_{X}+C\left(\dfrac{1}{\sqrt{\omega}}+\dfrac{1}{\omega}\right)(\left\|a\right\|_{X}+\left\|b\right\|_{X})\left\|y-z\right\|_{X},
\end{equation}
for $y,z\in X_\rho$. By hypothesis $$\|\varphi\|_{C_1(\R)}:=\max\{\sup_{t\in \R}|\varphi(t)|,\sup_{t\in \R}|\varphi'(t)|\}\leq \epsilon,$$ since $a(x,t)=-A(x)\varphi(t)$ we get
$$\|a(\cdot,t)\|_{H^2(I)}^2=(\varphi(t))^2\int_0^1 (A^2+A_x^2+A_{xx}^2)dx=(C\varphi(t))^2,$$
with $C^2:={\displaystyle\int_0^1 (A^2+A_x^2+A_{xx}^2)dx}$.  So
$$\|a\|_{X}={\displaystyle \sup_{t\in \R}\{
|\varphi(t)|C\}\leq C\|\varphi\|_{C_1(\R)}\leq C\epsilon. }$$
Analogously,  taking $b(x,t)=-A_x(x)\varphi(t)$, ensures that
$$\|b(\cdot,t)\|_{H^2(I)}^2=(\varphi(t))^2\int_0^1 (A_x^2+A_{xx}^2+A_{xxx}^2)dx=(C\varphi(t))^2.$$
Thus,
$$\|b\|_{X}={\displaystyle \sup_{t\in \R}\{
|\varphi(t)|C\}\leq C\|\varphi\|_{C_1(\R)}\leq C\epsilon. }$$
As $f(x,t)=A(x)\varphi'(t)+A'(x)\varphi(t)-A(x)A'(x)\varphi^2(t)$ and $0<\epsilon\leq 1,$ follows that $$\|f\|_{X}\leq C\epsilon,$$
where $C>0$ independent of $\epsilon.$  

Finally, thanks to the previous inequality, let us consider $$\rho= \dfrac{3C}{\omega}\left\|f\right\|_{X}\leq \dfrac{3C^2}{\omega}\epsilon.$$ For $\epsilon\ll 1$ small enough we have 
$$\rho\ll 1,$$
$$C\left(\dfrac{1}{\sqrt{\omega}}+\dfrac{1}{\omega}\right)\rho\leq \dfrac{1}{3},$$
and
$$C\left(\dfrac{1}{\sqrt{\omega}}+\dfrac{1}{\omega}\right)(\left\|a\right\|_{X}+\left\|b\right\|_{X})<\dfrac{1}{3}.$$
Therefore,  thanks to \eqref{eq25} and \eqref{eq26}, the following holds true
$$
\|\Psi y\|_{X}\leq\rho
$$
and
$$
\|\Psi y-\Psi z\|_{X}\leq\dfrac{2}{3}\|y-z\|_X,
$$
respectively.  Therefore,  using the Banach fixed-point theorem, there exists a unique $y\in X_\rho$ such that $\Psi y=y.$ Thus, for such $y$ yields that $$\|y\|_X=\|\Psi y\|_X\leq \rho$$ and $y$ is a mild solution for \eqref{eq7}.  As $y(x,t)=u(x,t)+A(x)\varphi(t)$, we get 
$$\|u\|_X\leq C\epsilon,$$
showing the result.
\end{proof}

\section{Massera's theorems for the Kawahara operator}\label{Sec3}
In this section our goal in to present several Massera's type theorems associated for the system 
\begin{equation}\label{eq1_a}
\begin{cases}
u_t+u_{xxx}-u_{xxxxx}+uu_x=0, & (x,t)\in I\times \R,\\
u(0,t)=\varphi(t), \ u(1,t)=u_x(1,t)=u_x(0,t)=0, &  t\in \R,\\
u_{xx}(1,t)=\alpha u_{xx}(0,t), & t\in \R,\\
u(x,0)=u_0(x), &x\in I.
\end{cases}
\end{equation}
 These theorem ensures that this higher order dispersive equation has recurrent solutions.  Let us start proving the first result in this way.

\subsection{Proof of Theorem \ref{main3}} 
We have that $v(x,t)=u(x,t+T)$  is the unique solution of 
\begin{equation*}
\begin{cases}
v_t+v_{xxx}-v_{xxxxx}+vv_x=0,& (x,t)\in I\times \R,\\
v(0,t)=\varphi(t+T)=\varphi(t), &  t\in \R,\\
v(1,t)=v_x(1,t)=v_x(0,t)=0, & t\in \R,\\
v_{xx}(1,t)=\alpha v_{xx}(0,t),& t\in \R.
\end{cases}
\end{equation*}
The system above is exactly \eqref{eq1_a},  so the uniqueness of solutions give us that
$$u(x,t+T)=v(x,t)=u(x,t),$$
for all $(x,t)\in I\times \R,$ showing the result. \qed

\subsection{Quasi-periodic solution}
In this section, we are interested in analysing the quasi-periodic solutions of \eqref{eq1_a}. Before it, we present some definitions necessary for this study. 
\begin{definition} We say that the real numbers $\omega_1,\omega_2,\dots,\omega_k$ are rationally independent when
$$m_1\omega_1+\dots+m_k\omega_k=0,$$
only happens when $m_1=\cdots m_k =0$,  with $m_1,\dots,m_k\in Q $,  where $Q$ is the set of all rational numbers. 
\end{definition}

Let $f:I\times \R\longrightarrow\R$ be a continuous function and $e_i$ a unitary vector of $\R^k$ such that the i-$th$ component is $1$ and the others are zero.  We have the following definition.
\begin{definition}
A function $f(x,t)$ is denoted by $\overline{\omega}$--quasi-periodic in $t$ uniformly with respect to $x\in I,$ if there are $\omega_1,\dots,\omega_k\in \R$ rationally independent and a function $F(x,u)\in C(I\times \R^k,\R)$ such that
$$f(x,t)=F(x, t\omega_1,t\omega_2,\dots, t\omega_k)=F(x,t\overline{\omega}),\quad \forall t\in \R\mbox{ and }x\in I,$$
where $\overline{\omega}=(\omega_1,\dots,\omega_k)$ and $F(x,u+2\pi e_i)=F(x,u)$, for all $u\in \R^k$ and $x\in I, \ i=1,2,\dots, k.$ The numbers $\omega_1,\dots,\omega_k$ are called basic frequencies of $f.$ 
\end{definition}

\begin{definition}
We say that  $\varphi(t)$ is $\overline{\omega}$--quasi-periodic in  $t$ if there is $\Phi(u)\in C(\R^k,\R)$ such that
$$\Phi(u+2\pi e_i)=\Phi(u), \quad \forall u\in \R^k, i=1,\dots, k$$
and
$$\varphi(t)=\Phi(t\overline{\omega})=\Phi(t\omega_1,\dots, t\omega_k) \quad \forall t\in \R.$$
\end{definition}

With these definitions in hand,  we prove now the second main result of the work. 
\begin{proof}[Proof of Theorem \ref{main4}] Since $\varphi(t)$ is $\overline{\omega}$--quasi-periodic function in $t$, by definition,  there exists $\Phi(u)\in C(\R^k,\R)$ such that
$$\Phi(u+2\pi e_i)=\Phi(u), \forall u\in \R^k, \ i=1,\dots, k$$
and
$$\varphi(t)=\Phi(t\overline{\omega})=\Phi(t\omega_1,\dots, t\omega_k), \quad \forall t\in \R.$$
Set $\overline{\alpha}=(\alpha_1,\dots,\alpha_k)\in \R^k$ and $\varphi_{\overline{\alpha}}(t)=\Phi(\overline{\alpha}+t\overline{\omega}).$ Thanks to the Theorem \ref{main2}, for each boundary force $\varphi_{\overline{\alpha}}(t)$ the equation \eqref{eq1_a} has unique solution $u_{\overline{\alpha}}\in X_{C_\epsilon}.$ 

Pick now
\begin{equation}
U(x,\overline{\alpha}):=u_{\overline{\alpha}}(x,0).
\label{eq29}
\end{equation}
Thus, $U$ is well defined due to the uniqueness of the solutions.  We prove the result by several claims.  

\vspace{0.2cm}
\noindent\textbf{Claim 1.} $u_{\overline{\alpha}}(x,t+h)=u_{\overline{\alpha}+h\overline{\omega}}(x,t).$
\vspace{0.2cm}

Indeed,  noting that
$$\varphi_{h\overline{\omega}+\overline{\alpha}}(t)=\Phi(t\overline{\omega}+h\overline{\omega}+\overline{\alpha})=\Phi((t+h)\overline{\omega}+\overline{\alpha})=\varphi_{\overline{\alpha}}(t+h),$$
we have that $u_{\overline{\alpha}}(x,t+h)$ and $u_{\overline{\alpha}+h\overline{\omega}}(x,t)$ are solutions of \eqref{eq1_a} with boundary force $\varphi_{h\overline{\omega}+\overline{\alpha}}(t)$. The uniqueness of solutions ensures that 
$$u_{\overline{\alpha}}(x,t+h)=u_{\overline{\alpha}+h\overline{\omega}}(x,t),$$
and the Claim 1 is proved.

\vspace{0.2cm}
\noindent\textbf{Claim 2.} $U(x, \overline{\alpha})=U(x,\alpha_1,\dots,\alpha_k)$ has period $2\pi$ with respect to each argument $\alpha_i.$ 
\vspace{0.2cm}

In fact,  taking $t=0$,  in the Claim 1, we get
$$u_{\overline{\alpha}}(x,h)=u_{\overline{\alpha}+h\overline{\omega}}(x,0)=U(x,\overline{\alpha}+h\overline{\omega}),\quad \forall h\in \R.$$
As $h$ is arbitrary,  we have that 
$$u_{\overline{\alpha}}(x,t)=U(x,\overline{\alpha}+t\overline{\omega}),\quad \forall t\in \R.$$
So,  \eqref{eq29} help us to ensure that
$$U(x,\overline{\alpha}+2\pi e_i)=u_{\overline{\alpha}+2\pi e_i}(x,0),$$
where $\{e_1,\dots, e_k\}$ is the standard basis in $\R^k.$ Since
$$\varphi_{\overline{\alpha}+2\pi e_i}(t)=\Phi(\overline{\omega}+\overline{\alpha}+2\pi e_i)=\Phi(\overline{\omega}+\overline{\alpha})=\varphi_{\overline{\alpha}}(t),$$
holds true, the uniqueness of solution guarantied by Theorem \ref{main2}, gives $$u_{\overline{\alpha}+2\pi e_i}=u_{\overline{\alpha}},$$ therefore
$$U(x,\overline{\alpha}+2\pi e_i)=u_{\overline{\alpha}+2\pi e_i}(x,0)=u_{\overline{\alpha}}(x,0)=U(x,\overline{\alpha}),$$
and the Claim 2 is showed. 

Finally,  taking $\overline{\alpha}=(0,\dots, 0)\in \R^k,$ we have that the external force $\varphi_{\overline{\alpha}}(t)=\varphi(t)$ and $u(x,t)=U(x,t\overline{\omega}). $ Thus, we get  $u$ is a $\overline{\omega}$--quasi-periodic solution in $t.$ 
\end{proof}

\subsection{Almost periodic solution} In this section the goal is to prove that \eqref{eq1_a} have almost periodic solutions.  To do that, let us begin this subsection with the following definition. 
\begin{definition}
Let $(Y, d)$ be a separable and complete metric space and $f:\R\longrightarrow Y$ be a continuous mapping. The function $f$ is said to be almost periodic if for every $\delta>0$ there exists a constant $l(\delta)>0$ such that any interval of length $l(\delta)$ contains at least a number $\tau$ for which $$\sup_{t\in\R} d(f(t+\tau),f(t))<\delta.$$
\end{definition}

Now, we are in position to prove the last result of the work.
\begin{proof}[Proof of Theorem \ref{main5}]
Consider $y,a,b$ and $f$ satisfying \eqref{eq7}, \eqref{eq7a} and \eqref{eq7b}.  Straightforward calculation, thanks to the following change of variable $\tau=s-\sigma$, shows that 
\begin{equation}
\label{eq30}
\begin{split}
\|y(t+\sigma)-y(t)\|_{H^2(I)}=&\left\|\int_{-\infty}^{t+\sigma}S(t+\sigma-s)(-yy_x-ay_x-by+f)(s)ds\right.\\
& \left.-\int_{-\infty}^{t}S(t-s)(-yy_x-ay_x-by+f)(s)ds\right\|_{H^2(I)}\\
\leq& \left\|\int_{-\infty}^{t}S(t-s)\left(y(s+\sigma)y_x(s+\sigma)-y(s)y_x(s)\right)ds\right\|_{H^2(I)}\\
&+\left\|\int_{-\infty}^{t}S(t-s)\left(a(s+\sigma)y_x(s+\sigma)-a(s)y_x(s)\right)ds\right\|_{H^2(I)}\\
& +\left\|\int_{-\infty}^{t}S(t-s)\left(b(s+\sigma)y(s+\sigma)-b(s)y(s)\right)ds\right\|_{H^2(I)}\\
& + \left\|\int_{-\infty}^{t}S(t-s)\left(f(s+\sigma)-f(f)\right)ds\right\|_{H^2(I)}.
\end{split}
\end{equation}
Set  $z(\cdot, s)=y(\cdot, s+\sigma)$ in the expression \eqref{eq19}, we get
\begin{equation}
\label{eq31}
\left\|\int_{-\infty}^{t}S(t-s)\left(y(s+\sigma)y_x(s+\sigma)-y(s)y_x(s)\right)ds\right\|_{H^2(I)}\leq \ C\left(\dfrac{1}{\sqrt{\omega}}+\dfrac{1}{\omega}\right)\|y(\cdot+\sigma)-y(\cdot)\|_X^2.
\end{equation}
Therefore, follows by \eqref{eq16} and \eqref{eq14} with $a(\cdot+\sigma)-a(\cdot)$ instead of $a$ that
\begin{equation}\label{eq32}
\begin{split}
\left\|\int_{-\infty}^{t}S(t-s)\left(a(s+\sigma)y_x(s+\sigma)-a(s)y_x(s)\right)ds\right\|_{H^2(I)} \leq& C\left(\dfrac{1}{\sqrt{\omega}}+\dfrac{1}{\omega}\right)\|a\|_X\|y(\cdot+\sigma)-y(\cdot)\|_X\\& +C\left(\dfrac{1}{\sqrt{\omega}}+\dfrac{1}{\omega}\right)\|a(\cdot+\sigma)-a(\cdot)\|_X\|y\|_X,
\end{split}
\end{equation}
In an analogous way, we get 
\begin{equation}
\label{eq33}
\begin{split}
\left\|\int_{-\infty}^{t}S(t-s)\left(b(s+\sigma)y(s+\sigma)-b(s)y(s)\right)ds\right\|_{H^2(I)} \leq & C\left(\dfrac{1}{\sqrt{\omega}}+\dfrac{1}{\omega}\right)\|b\|_X\|y(\cdot+\sigma)-y(\cdot)\|_X \\&+C\left(\dfrac{1}{\sqrt{\omega}}+\dfrac{1}{\omega}\right)\|b(\cdot+\sigma)-b(\cdot)\|_X\|y\|_X, 
\end{split}
\end{equation}
where we have used \eqref{eq17} and \eqref{eq15} with $b(\cdot+\sigma)-b(\cdot)$ instead of $b.$ Due to the Theorem \ref{main1} and using the change of variables $\tau=t-s$, yields that
\begin{equation}
\label{eq34}
\left\|\int_{-\infty}^t S(t-s)\left(f(\cdot,s+\sigma)-f(\cdot,s)\right)ds\right\|_{H^2(I)}\leq  \dfrac{C}{\omega}\left\|f(\cdot+\sigma)-f(\cdot)\right\|_{X}.
\end{equation}

Now, replacing  \eqref{eq31}, \eqref{eq32},  \eqref{eq33} and \eqref{eq34} into \eqref{eq30}, we ensures that 
\begin{equation}
\label{eq35}
\begin{split}
\|y(\cdot+\sigma)-y(\cdot)\|_{X}\leq&C\left(\dfrac{1}{\sqrt{\omega}}+\dfrac{1}{\omega}\right)\|y(\cdot+\sigma)-y(\cdot)\|_X^2\\
&+C\left(\dfrac{1}{\sqrt{\omega}}+\dfrac{1}{\omega}\right)(\|a\|_X+\|b\|_X)\|y(\cdot+\sigma)-y(\cdot)\|_X\\
& +C\left(\dfrac{1}{\sqrt{\omega}}+\dfrac{1}{\omega}\right)\left(\|a(\cdot+\sigma)-a(\cdot)\|_X+\|b(\cdot+\sigma)-b(\cdot)\|_X\right)\|y\|_X\\
&+ \dfrac{C}{\omega}\left\|f(\cdot+\sigma)-f(\cdot)\right\|_{X}.
\end{split}
\end{equation}
Thus, taking $y\in X_\rho$ and $0<\epsilon\ll 1$ in the proof of Theorem \ref{main2} such that 
$$2C\left(\dfrac{1}{\sqrt{\omega}}+\dfrac{1}{\omega}\right)\rho<\dfrac{1}{3}$$
and
$$C\left(\dfrac{1}{\sqrt{\omega}}+\dfrac{1}{\omega}\right)(\|a\|_X+\|b\|_X)<\dfrac{1}{3},$$
we have that Theorem \ref{main2} is still valid and also is verified that
$$C\left(\dfrac{1}{\sqrt{\omega}}+\dfrac{1}{\omega}\right)\|y(\cdot+\sigma)-y(\cdot)\|_X^2\leq 2\rho C\left(\dfrac{1}{\sqrt{\omega}}+\dfrac{1}{\omega}\right)\|y(\cdot+\sigma)-y(\cdot)\|_X<\dfrac{1}{3}\|y(\cdot+\sigma)-y(\cdot)\|_X.$$
Finally,  applying it in \eqref{eq35} we have
\begin{equation}
\label{eq36}
\begin{split}
\|y(\cdot+\sigma)-y(\cdot)\|_{X}\leq&2\rho C\left(\dfrac{1}{\sqrt{\omega}}+\dfrac{1}{\omega}\right)\|y(\cdot+\sigma)-y(\cdot)\|_X\\
&+C\left(\dfrac{1}{\sqrt{\omega}}+\dfrac{1}{\omega}\right)(\|a\|_X+\|b\|_X)\|y(\cdot+\sigma)-y(\cdot)\|_X\\
& +C\rho\left(\dfrac{1}{\sqrt{\omega}}+\dfrac{1}{\omega}\right)\left(\|a(\cdot+\sigma)-a(\cdot)\|_X+\|b(\cdot+\sigma)-b(\cdot)\|_X\right)\\
&+  \dfrac{C}{\omega}\left\|f(\cdot+\sigma)-f(\cdot)\right\|_{X}\\
\leq&\dfrac{1}{3}\|y(\cdot+\sigma)-y(\cdot)\|_X+\dfrac{1}{3}\|y(\cdot+\sigma)-y(\cdot)\|_X\\
& +\dfrac{1}{3}\left(\|a(\cdot+\sigma)-a(\cdot)\|_X+\|b(\cdot+\sigma)-b(\cdot)\|_X\right)\\
&+\displaystyle  \dfrac{C}{\omega}\left\|f(\cdot+\sigma)-f(\cdot)\right\|_{X}\\
\end{split}
\end{equation}
and follows that
\begin{equation}
\label{eq37}
\|y(\cdot+\sigma)-y(\cdot)\|_{X}\leq \|a(\cdot+\sigma)-a(\cdot)\|_X+\|b(\cdot+\sigma)-b(\cdot)\|_X+\dfrac{3C}{\omega}\left\|f(\cdot+\sigma)-f(\cdot)\right\|_{X}.
\end{equation}
Since $a$, $b$, $f$ are almost periodic functions, $y$ is also an almost periodic function. According to the fact that $$y(x,t):=u(x,t)-A(x)\varphi(t),$$ we obtain that $u$ is also an almost periodic function. Thus, $u$ is almost periodic solution of \eqref{eq1_a} and the Theorem is achieved. 
\end{proof}

\section{Further comments}\label{Sec4}
In this work were able to present properties for a higher order dispersive system, namely, the Kawahara equation,  posed on a bounded domain. Many results in the literature, as we saw in the introduction, treated this equation of control point of view. Here, we provide periodic properties for the following initial boundary value problem 
\begin{equation}\label{eq1f}
\begin{cases}
u_t+u_{xxx}-u_{xxxxx}+uu_x=0, & (x,t)\in I\times \R\\
u(0,t)=\varphi(t), \ u(1,t)=u_x(1,t)=u_x(0,t)=0,& t\in \R\\
u_{xx}(1,t)=\alpha u_{xx}(0,t), & t\in \R\\
u(x,0)=u_0(x), &x\in I,
\end{cases}
\end{equation}
with a forcing boundary term $\varphi(t)$ and a term $\alpha u_{xx}(0,t)$ acting as a damping mechanism.  Thus,  we have succeeded to prove Massera-type theorems for the solution of \eqref{eq1f}.   With respect to the generality of the work, let us make some additional comments.
\begin{itemize}
\item Theorems \ref{main3}, \ref{main4} and \ref{main5} can be obtained for more general nonlinearities. Indeed, if we consider $u \in \mathcal{B}:=C(0,T;L^2(0,1))\cap\ L^2(0,T;H^2(0,1))$ and the nonlinearity $u^pu_x$, $p \in (2,4]$, we have that 
$$
\int_{0}^{T} \int_{0}^{1} |u^{p+2} |d x d t\leqslant C\left\|u\right\|_{C([0,T];L^{2}(0,1))}^{p} \int_{0}^{T}\left\|u_{x}\right\|^{2} d t \leqslant C\left\|u\right\|_{\mathcal{B}}^{p + 2},
$$
by the Gagliardo–Nirenberg inequality. Moreover, recently, Zhou \cite{Zhou} showed the well-posedness of the following initial boundary value problem
\begin{equation}\label{int2} 
\left\{\begin{array}{lll}
u_t-u_{xxxxx}=c_{1} uu_x+c_{2} u^{2}u_x+b_{1}  u_xu_{xx}+b_{2} uu_{xxx},& x \in(0, L), \ t \in \mathbb{R}^{+}, \\
u(t, 0)=h_{1}(t), \quad u(t, L)=h_{2}(t),  \quad u_x(t, 0)=h_{3}(t), &  t \in \mathbb{R}^{+},\\
 u_x(t, L)=h_{4}(t), \quad u_{xx}(t, L)=h(t), &  t \in \mathbb{R}^{+},\\
u(0, x)=u_{0}(x),& x \in(0, L),
\end{array}\right.
\end{equation}
Thus, due to the previous inequality and the results proved in \cite{Zhou}, when we consider $b_1=b_2=0$ and the combination $c_1uu_{x}+c_2u^2u_x$ instead of $uu_x$ on \eqref{eq1f}, the main results of this works remains valid.
\vspace{0.1cm}
\item As in the classical framework of the Massera's theorem,  a principal point is to prove that the initial boundary value problem \eqref{eq1f} admits bounded solution, to do that, an important step is the study of the energy associated with the linear system under consideration, this analysis was made in the Appendix \ref{appendix}. 
\vspace{0.1cm}
\item An important point of the previous remark is to deal with the energy of \eqref{eq1f} we analyse the Kawahara operator removing the drift term $u_x$. This term presents an extra problem because a critical set appears, see \cite{CaKawahara} for details. In this way, to overcome this difficulty it was necessary to remove the drift term. Thus,  an interesting open problem is to extend the result presented in this paper for the Kawahara equation \eqref{eq1f} with the drift term taking into accounte that this equation, with $\varphi(t)=0$, has the critical set phenomena,  as conjectured in \cite{CaKawahara}.
\vspace{0.1cm}
\item It is important to point out that the Massera-type theorem has been extended to many differential equations as we can see in \cite{MoChen1,Fleury, hernandez, Liu, zubelevich} and the references therein. The method employed in these works is to prove the existence of periodic solutions if the solution of the equation under consideration is bounded. 
\vspace{0.1cm}
\item Finally, there are two important points related with the Massera-type theorems for the Kawahara equation. The first one is that we can work with more general nonlinearities, as mentioned before.  The second one is the strong relation between the damping mechanism (stabilization problem) and the Massera-type theorems in our case.
\end{itemize}

\appendix 

\section{Additional properties} \label{appendix}
In this appendix we present some additional properties of the linear Kawahara system. For sake of simplicity we present the results for the linear system, however,  the results obtained here can be also extended for the nonlinear system.  Precisely,  let us study the energy properties for the following linear system
\begin{equation}\label{eq1aa}
\begin{cases}
u_t+u_{xxx}-u_{xxxxx}=0, & (x,t)\in I\times \R,\\
u(0,t)=u(1,t)=u_x(1,t)=u_x(0,t)=0, &  t\in \R,\\
u_{xx}(1,t)=\alpha u_{xx}(0,t), & t\in \R,\\
u(x,0)=u_0(x), &x\in I,
\end{cases}
\end{equation}
where $|\alpha|<1$.  Note that multiplying \eqref{eq1aa} by $u$ and integrating over
$\left(  0,L\right)$ yields
\begin{equation}
\dfrac{d}{dt}\int_{0}^{L}|u(x,t)|^{2}\,dx=\dfrac{1}{2}(\alpha^2-1)(u_{xx}(0))^2\leq0, \  \forall t\geq 0.
\label{G1.1}
\end{equation}
This indicates that the energy $E(t)=\frac12\|u\|^{2}(t)$ associated with \eqref{eq1aa} is not increasing, and the term $\alpha u_{xx}(0,t)$ designs a damping mechanism. To ensure that this energy decays exponentially is natural to show an \textit{observability inequality} associated with the solutions of \eqref{eq1aa}.  Before presenting it, let us first prove \textit{a weak observability inequality}.
\begin{proposition}
\label{prop7}
Consider $u$ solution of \eqref{eq1aa} belonging in $C(0,T;L^2(0,1))\cap L^2(0,T;H^2(0,1))$. Thus, we have
\begin{equation}
\label{eq46}
\dfrac{1}{2}\|u_0\|^2\leq \dfrac{1}{2T}\int_{0}^{T}\int_0^1 |u(x,t)|^2dxdt +\dfrac{1-\alpha^2}{2}\int_{0}^{T} |u_{xx}(0,t)|^2dt,
\end{equation}
for all $T>0$.
\end{proposition}
\begin{proof}
We prove the result for the initial data $u_0\in D(A)$, the result in $L^2(0,1)$ follows by density.  First,  multiplying the system \eqref{eq1aa} by $(T-t)u$, integrating by parts in $(0,T)\times (0,1)$ and using the boundary conditions we have
$$-\dfrac{T}{2}\|u_0\|^2+ \dfrac{1}{2}\int_{0}^{T}\int_0^1 u^2 dxdt+ \dfrac{1}{2}(\alpha^2-1)\int_0^T(T-t)(u_{xx}(0,t))^2dt=0,$$
or equivalently, 
$$\dfrac{1}{2}\int_{0}^{T}\int_0^1 u^2 dxdt+ \dfrac{1}{2}(\alpha^2-1)\int_0^T(T-t)(u_{xx}(0,t))^2dt=\dfrac{T}{2}\|u_0\|^2.$$ 
Thus, we get 
$$ \dfrac{1}{2}\|u_0\|^2\leq \dfrac{1}{2T}\int_{0}^{T}\int_0^1 u^2 dxdt+ \dfrac{(\alpha^2-1)}{2}\int_0^T|u_{xx}(0,t)|^2dt,$$
showing the proposition.
\end{proof}

Now, we are in position to prove that the energy associated to \eqref{eq1aa} decays exponentially.
\begin{theorem}
\label{teo6}
There exists $C>0$ and $\mu>0$ such that
\begin{equation}
\label{eq51}
E(t)\leq C\|u_0\|^2e^{-\mu t},
\end{equation}
for all $t\geq 0$ and $u$ solution of \eqref{eq1aa} with $u_0\in L^2(0,1).$
\end{theorem}
\begin{proof}
This results is consequence of the following observability inequality
\begin{equation}
\label{eq52}
\dfrac{1}{2}\int_{0}^{T}\int_0^1 u^2 dxdt\leq c_1\int_0^T|u_{xx}(0,t)|^2dt,
\end{equation}
for some constant $c_1>0$ independent of the solution $u.$ 

In fact,  replacing \eqref{eq52} in \eqref{eq46}, we get
\begin{equation}
\label{eq62}
\begin{split}
\dfrac{1}{2}\|u_0\|^2\leq& \dfrac{1}{2T}\int_{0}^{T}\int_0^1 |u(x,t)|^2dxdt +\left(\dfrac{1-\alpha^2}{2}\right)\int_{0}^{T} |u_{xx}(0,t)|^2dt\\
\leq& \dfrac{1}{2T}\int_{0}^{T} |u_{xx}(0,t)|^2dt+\left(\dfrac{1-\alpha^2}{2}\right)\int_{0}^{T} |u_{xx}(0,t)|^2dt\\
=&C\int_{0}^{T} |u_{xx}(0,t)|^2dt,
\end{split}
\end{equation}
where $C=C(T, \alpha)>0$.  As we have that 
$$E'(t)=\dfrac{1}{2}(\alpha^2-1)(u_{xx}(0,t))^2\leq 0,$$
integrating in $(0,t)$ the previous equation and multiplying by $(1+C),$ where $C$ is the same constant obtained previously,  we have that
\begin{equation}
\label{eq63}
(1+C)E(T)\leq CE(0).
\end{equation}
Thus,
$$E(T)\leq \gamma E(0),  \mbox{ where } 0<\gamma=\dfrac{C}{1+C}<1,$$
with $\gamma:=\gamma(T,\alpha)$.  Now,  the same argument used on the interval $[(m-1) T, m T]$ for $m=1,2, \ldots$, yields that
$$
E(m T) \leq \gamma E((m-1) T) \leq \cdots \leq \gamma^{m} E(0)
$$
Thus, we have
$$
E(m T) \leq e^{-\nu m T} E(0)
$$
with
$$
\nu=\frac{1}{T} \ln \left(1+\frac{1}{C}\right)>0
$$
For an arbitrary positive $t$, there exists $m \in \mathbb{N}^{*}$ such that $(m-1) T<t \leq m T$, and by the non-increasing property of the energy, we conclude that
$$
E(t) \leq E((m-1) T) \leq e^{-\nu(m-1) T} E(0) \leq \frac{1}{\gamma} e^{-\nu t} E(0),
$$
showing the result. 
 \end{proof}
 
 Let us now prove the observability inequality.
 \begin{proof}[Proof of  \eqref{eq52}]
We argue by contradiction. Suppose that \eqref{eq52} does not holds. Thus, there exist a sequence $\{u_n\}_{n\in\mathbb{N}}$ of \eqref{eq1aa} such that 
\begin{equation}
\label{eq53}
 \lim_{n\longrightarrow +\infty}\dfrac{\displaystyle{\int_{0}^{T}\int_0^1 u^2 dxdt}}{\displaystyle{\int_0^T|u_{xx}(0,t)|^2dt}}=+\infty.
\end{equation}
Now on,  taking  $\lambda_n=\sqrt{\displaystyle{\int_{0}^{T}\int_0^1 |u_n(x,t)|^2 dxdt}}$ and $v_n(x,t)=\dfrac{u_n(x,t)}{\lambda_n},$ we have that $\{v_n\}_{n\in\mathbb{N}}$ is a sequence satisfying \eqref{eq1aa} with initial data $v_n(x,0)=\dfrac{u_n(x,0)}{\lambda_n}$ and
\begin{equation}
\label{eq54}\int_{0}^{T}\int_0^1 |v_n(x,t)|^2 dxdt=\dfrac{1}{\lambda_n^2}\int_{0}^{T}\int_0^1 |u_n(x,t)|^2 dxdt=1.
\end{equation}
Thanks to the equation \eqref{eq53}, we have
\begin{equation}
\label{eq55}
\begin{split}
\lim_{n\rightarrow +\infty}\int_{0}^{T}|(v_n)_{xx}(0,t)|^2 dt=&\lim_{n\rightarrow +\infty}\dfrac{\displaystyle{\int_{0}^{T}|(u_n)_{xx}(0,t)|^2 dt}}{\lambda_n^2}\\
=&\lim_{n\rightarrow +\infty}\dfrac{\displaystyle{\int_{0}^{T}|(u_n)_{xx}(0,t)|^2 dt}}{\displaystyle{\int_{0}^{T}\int_0^1 |u_n(x,t)|^2 dxdt}}=0.
\end{split}
\end{equation}
Due the relation \eqref{eq46}, since \eqref{eq54} and \eqref{eq55} are verified, we have that  $\{v_n(x,0)\}_{n\in\mathbb{N}}$ is a sequence bounded in $L^2(0,1).$ Therefore,  Propositions \ref{c_0} and \ref{prop2} gives the existence of a constant $M>0$ such that 
\begin{equation}
\label{eq56}
\|v_n\|^2_{L^2(0,T;H^2_0(0,1))}\leq M, 
\end{equation}
for all $n\in\mathbb{N}$. Since $H^2_0(0,1)\hookrightarrow L^2(0,1)$ compactly, we have  $\{v_n\}_{n\in\mathbb{N}}$ is  relatively compact in $L^2(0,T; L^2(0,1)).$ Thus,  there exist a subsequence, still denoted by $\{v_n\}_{n\in\mathbb{N}}$, such that 
$$v_n\rightharpoonup v\mbox{ weakly in } L^2(0,T; H^2_0(0,1))$$
Moreover, since $v_{n,t}$ is bounded in $L^2(0,T; H^{-3}(0,1))$, so thanks to the Aubin-Lions's theorem we have
\begin{equation}
\label{eq57}
v_n\rightarrow v\mbox{ strongly in } L^2(0,T; L^2(0,1)).
\end{equation}
By \eqref{eq54}, we get 
\begin{equation}
\label{eq58}
\|v\|_{L^2(0,T; L^2(0,1))}=1,
\end{equation} 
and so using \eqref{eq55} and \eqref{eq57}, verifies that
\begin{equation}
\label{eq59}{\displaystyle \int_{0}^{T}|v_{xx}(0,t)|^2 dt\leq\liminf_{n\rightarrow +\infty}\int_{0}^{T}|(v_n)_{xx}(0,t)|^2 dt=0},
\end{equation}
which ensures $v_{xx}(0,t)= 0$ , for all  $t\in (0,T).$  Therefore, the function $v$ satisfies
\begin{equation}
\begin{cases}
v_t+v_{xxx}-v_{xxxxx}=0,& (x,t)\in I\times \R,\\
v(0,t)= v(1,t)=v_x(1,t)=v_x(0,t)=v_{xx}(1,t)= v_{xx}(0,t)=0, & t\in \R,\\
v(x,0)=v_0, & x\in I.
\end{cases}
\label{eq61}
\end{equation}
The result follows by using \cite[Lemma 1.1]{ERRATUM} that give us $v=0$, contradicting the hypotheses \eqref{eq58}. Thus,  the observality inequality holds true.  \end{proof}

\subsection*{Acknowledgments:} 
Capistrano–Filho was supported by CNPq grant number 307808/2021-1,  CAPES grant numbers 88881.311964/2018-01 and 88881.520205/2020-01,  MATHAMSUD grant 21-MATH-03 and Propesqi (UFPE).  This work is part of the PhD thesis of I.  M. de Jesus at Department of Mathematics of the Universidade Federal de Pernambuco.


\begin{thebibliography}{99}  
\bibitem{ASL} B. Alvarez-Samaniego and D. Lannes, \emph{Large time existence for 3D water-waves and asymptotics}, Invent. Math. 171, no. 3, 485--541 (2008).                                                                               

 \bibitem{CaKawahara}F. D. Araruna, R. A. Capistrano-Filho and G. G. Doronin, \textit{Energy decay for the modified Kawahara equation posed in a bounded domain}, J. Math. Anal. Appl., 385:(2),  743--756 (2012). 
   

\bibitem{BLS}J. L. Bona, D. Lannes and J.-C. Saut, \emph{Asymptotic models for internal waves}, J. Math. Pures Appl. (9):89, no. 6, 538--566 (2008).

\bibitem{boumediene} B. Chentouf,  \textit{Well-posedness and exponential stability of the Kawahara equation with a time-delayed localized damping}, Mathematical Methods in the Applied Sciences,  https://doi.org/10.1002/mma.8369.

\bibitem{CaVi} R. A. Capistrano-Filho and V. H. Gonzalez Martinez, \textit{Stabilization results for delayed fifth order KdV-type equation in a bounded domain},  arXiv:2112.14854 [math.AP].

\bibitem{CaGo} R. A. Capistrano-Filho and M. M. de S. Gomes,  \textit{Well-posedness and controllability of Kawahara equation in weighted Sobolev spaces}, Nonlinear Analysis, Volume \textbf{207}, 1--24 (2021).

\bibitem{CaSo} R. A. Capistrano-Filho and L. S. de Sousa, \textit{Control results with overdetermination condition for higher order dispersive system}, Journal of Mathematical Analysis and Applications 506:(1), 1--22  (2022).

\bibitem{CaSoGa}R. A. Capistrano-Filho, L. S. de Sousa and F. A. Gallego, \textit{Control of Kawahara equation with overdetermination condition: The unbounded cases},  arXiv:2110.08803 [math.AP].

\bibitem{MoChen} M. Chen, {\em Internal controllability of the Kawahara equation on a bounded domain}, Nonlinear Analysis, 185, 356--373 (2019).

\bibitem{MoChen1}M. Chen, \textit{Recurrent solutions of the Korteweg–de Vries equation with boundary force}, Indian J.  Pure Appl. Math.  53, 112--126 (2022). 

\bibitem{Fleury}M. Fleury, J. G. Mesquita and A. Slavík,  \textit{Massera's theorems for various types of equations with discontinuous solutions}, Journal of Differential Equations, 269:12, 11667--11693 (2020).

\bibitem{Hasimoto1970}  H. Hasimoto, \emph{Water waves}, Kagaku, 40,  401--408  [Japanese]  (1970). 

\bibitem{hernandez} E. M.  Hernández,  \textit{A Massera type criterion for a partial neutral functional differential equation}, Electronic Journal of Differential Equations (EJDE),  No. 40: 17, (2002).


\bibitem {Kawahara}T. Kawahara, \textit{Oscillatory solitary waves in dispersive media}, J. Phys. Soc. Japan,  33 , 260--264 (1972).

%


\bibitem{Lannes} D. Lannes, \textit{The water waves problem. Mathematical analysis and asymptotics}. Mathematical Surveys and Monographs, 188. American Mathematical Society, Providence, RI, 2013. xx+321 pp.

\bibitem{Liu} Q. Liu, N. V.  Minh, G.  Nguerekata and  R. Yuan, \textit{ Massera Type Theorems for Abstract Functional Differential Equations}, Funkcialaj Ekvacioj,  51: 3, 329--350 (2008).

\bibitem{massera}J. L.  Massera,  \textit{The existence of periodic solutions of systems of differential  equations,} Duke Math.  J.  17,  457--475 (1950).

\bibitem{Pazy} A. Pazy,  \emph{ Semigroups of Linear Operators and Applications to Partial Differential Equations}, Appl. Math. Sci., vol. 44, Springer, New York, 1983.


\bibitem{CV} P. N.  Silva and C. F. Vasconcellos, \emph{Stabilization of the linear Kawahara equation with localized damping},  Asymptotic Analysis. \textbf{58}, 229--252 (2008).

\bibitem{ERRATUM} P. N.  Silva and C. F. Vasconcellos, \emph{Erratum: Stabilization of the linear Kawahara equation with localized damping},  Asymptotic Analysis. \textbf{66}, 119--124 (2010).


               
\bibitem{Zhou} D. Zhou, \textit{Non-homogeneous initial-boundary-value problem of the fifth-order Korteweg-de Vries equation with a nonlinear dispersive term},   Journal of Mathematical Analysis and Applications, Volume 497, Issue 1, (2021).

\bibitem{zubelevich} O. E. Zubelevich,  \textit{A note on theorem of Massera},  Regular and Chaotic Dynamics,  vol. 11, no. 4,  475--481 (2006).

\end{thebibliography}
\end{document}